\pdfoutput=1
\documentclass[12pt]{article}
\usepackage[a4paper,textwidth=160mm,textheight=221mm]{geometry}
\usepackage{amsmath,amsthm}

\usepackage{graphicx}
\usepackage{amsfonts}
\usepackage{amssymb}
\usepackage{color}
\usepackage[colorlinks=true,citecolor=black,linkcolor=black,urlcolor=blue]{hyperref}
\usepackage{xcolor}
\usepackage{bm}
\usepackage{tikz}
\usetikzlibrary{arrows.meta,positioning,calc,math}

\theoremstyle{plain}
\newtheorem{theorem}{Theorem}
\newtheorem{lemma}[theorem]{Lemma}

\newtheorem{proposition}[theorem]{Proposition}
\newtheorem*{restatedthm}{\restatedthmname}
\newcommand{\restatedthmname}{}
\newenvironment{restated}[1]{\renewcommand{\restatedthmname}{Theorem~\ref{#1}}\begin{restatedthm}}{\end{restatedthm}}
\theoremstyle{definition}
\newtheorem{definition}[theorem]{Definition}
\newtheorem{example}[theorem]{Example}
\theoremstyle{remark}

\newcommand{\bN}{\mathbb{N}}
\newcommand{\bQ}{\mathbb{Q}}
\newcommand{\bG}{\mathbf{G}}
\newcommand{\bH}{\mathbf{H}}

\NewDocumentCommand{\vertexs}{O{black} O{1cm} O{above} m O{(0,0)} m O{}}{
\path let \p1 = #6, \p2=#5 in node[style={draw,fill,color=#1,circle,minimum size=6pt,inner sep=0},label={[text=#1, label distance=#2]#3:}] (v#7#4) at (\x1+\x2,\y1+\y2){};}
\NewDocumentCommand{\vertex}{O{black} O{-3pt} O{above} m O{(0,0)} m O{}}{
\path let \p1 = #6, \p2=#5 in node[style={draw,fill,color=#1,circle,minimum size=6pt,inner sep=0},label={[text=#1, label distance=#2]#3:#4}] (v#7#4) at (\x1+\x2,\y1+\y2){};}
\NewDocumentCommand{\vertexl}{O{black} O{0pt} O{0pt} m O{(0,0)} m O{}}{
\path let \p1 = #6, \p2=#5 in node[style={draw,fill,color=#1,circle,minimum size=6pt,inner sep=0},label={[text=#1, shift={(#2, #3)}]#4}] (v#7#4) at (\x1+\x2,\y1+\y2){};}

\title{Proper circular arc graphs are $e$-positive}

\author{Aarush Vailaya\thanks{Department of Mathematics, Massachusetts Institute of Technology, Massachusetts, USA. Email: \texttt{aarushv@mit.edu}.}}
\date{}
\hypersetup{pdftitle={Proper circular arc graphs are e-positive},pdfauthor={Aarush Vailaya}}

\begin{document}

\maketitle

\begin{abstract}
We prove an $e$-positive formula for the chromatic symmetric function of proper circular arc graphs solving the $q=1$ case of a conjecture of Ellzey and of Alexandersson and Panova. In doing so, we provide a new proof of the $e$-positivity of unit interval graphs, which alongside Guay-Paquet's reduction gives a new proof of the Stanley--Stembridge conjecture. We define color matrices, which count proper colorings, and tableau matrices, whose entries are nonnegative rational numbers and ratios of elementary symmetric functions. We prove the two matrices are related by a single family of change of basis matrices, which become invertible after restricting to finitely many colors, and we show the chromatic symmetric function of proper circular arc graphs comes from taking the trace of these matrices. Using Hikita's tableaux, this gives an explicit formula for the chromatic symmetric function of a proper circular arc graph as a weighted sum over tableaux whose first and last $k$ vertices lie in the same columns.

\end{abstract}

\section{Introduction}\label{section:introduction}
The chromatic symmetric function $X_G(\bm x)$ generalizes the chromatic
polynomial, encodes additional information about acyclic orientations, and
has connections to Hessenberg varieties \cite{dothessenberg,chromquasihessenberg}
and Hecke algebras \cite{chromhecke,chromcharacters}. The Stanley--Stembridge conjecture
\cite[Conjecture 5.1]{chromsym}, \cite[Conjecture 5.5]{stanstem} states that the chromatic symmetric function of the incomparability graph
of a $(\bm 3+\bm 1)$-free poset has nonnegative coefficients in the
elementary symmetric function basis. By a reduction due to Guay-Paquet
\cite[Theorem 5.1]{stanstemreduction}, it suffices to consider unit interval
graphs. Hikita \cite[Theorem 1.6 and Corollary 1.8]{stanstemproof} resolved
the conjecture by providing a formula in terms of probabilities associated
to standard Young tableaux. Griffin, Mellit, Romero, Weigl, and Wen gave another proof using
Macdonald expansions of chromatic quasisymmetric functions
\cite[Corollary~3.8]{macdonaldchrom}.

Shareshian and Wachs introduced the chromatic quasisymmetric function
$X_G(\bm x;q)$, showed that it is symmetric for natural unit interval
graphs, and conjectured that its $e$-coefficients are polynomials in $q$
with nonnegative coefficients \cite{chromposquasi}, which was recently proved by OpenAI \cite[Theorem 1.1]{openaichrom}.
Hwang conjectured a finer positivity, grouping the proper colorings
into classes that each contribute an $e$-positive symmetric function
\cite[Conjecture~3.10]{hwangflip}; this remains open. Ellzey generalized
this function to directed graphs and conjectured the same positivity for
proper circular arc digraphs \cite[Conjecture 1.4]{chromquasidi}.
Alexandersson and Panova independently conjectured the same positivity,
together with unimodality \cite[Conjecture 5.2]{lltcycles}.
At $q=1$, this conjectures that every proper circular arc graph is $e$-positive.
Several families built from cycles are known to be $e$-positive
\cite{graphtwin,cyclechains,cyclechordsepositive}. The stronger positivity in $q$ is known for classes like directed cycles
\cite[Theorem~6.1]{chromquasidi}, \cite{lltcycles}, but remains open for arbitrary proper circular arc digraphs.

A parallel question concerns vertical-strip LLT polynomials, whose
coloring sums allow equal colors on ordinary edges and impose strict
inequalities along marked edges. Alexandersson and Panova conjectured
that their circular versions are $e$-positive after the substitution
$q\mapsto q+1$ \cite[Conjecture 4.5]{lltcycles}. This was
recently proved by Alexandersson and Dai \cite[Theorem~B]{lltflip},
together with the LLT analogue of Hwang's refinement
\cite[Theorem~A]{lltflip}.

The chromatic symmetric function is multiplicative over connected
components, but gluing graphs at a vertex does not correspond to
multiplication of their chromatic symmetric functions. In
\cite{singlegluing}, Tom and the author gave a matrix formula for this
operation by assigning an infinite $e$-positive matrix $M_G$, constructed from Hikita's tableau, to each natural unit interval graph. Matrix multiplication $M_GM_H=M_{G+H}$ corresponds to gluing two graphs together at a single vertex to form a larger natural unit interval graph, and the trace $\operatorname{tr} M_G=X_{G^\circ}(\bm x)$ corresponds to identifying the first and last vertex of the graph to form a proper circular arc graph.

This paper generalizes the above result to multiple vertices. In particular, given a graph $G$ we define color
matrices, whose entries count proper colorings with specified colors on
the marked vertices, and $e$-positive tableau matrices, which record the columns of
specified boxes in Hikita's tableau. We show these matrices are related by a change of basis and satisfy the same multiplication and trace properties. Since every proper circular arc graph can be obtained by identifying the first and last $k$ vertices of a natural unit interval graph, the tableau
trace then proves that every proper circular arc graph is $e$-positive
(Theorem~\ref{thm:pca}), establishing the $q=1$ case of the
conjecture of Ellzey and of Alexandersson and Panova.

Moreover, the tableau matrices give an explicit formula for the chromatic symmetric function of a proper circular arc graph (Theorem~\ref{thm:pca-formula}). We cut the circle to obtain a natural unit interval graph, place the first $k$ vertices at the ends of rows of lengths $j_1,\dots,j_k$, build the rest of the graph with Hikita's rule, and keep the tableaux whose last $k$ vertices return to the columns $j_1,\dots,j_k$. Unlike Hikita's formula, this sums over one starting tableau for each choice of $j$, so the number of terms grows exponentially in $k$, but every term is explicit. In relating the two matrices, we also provide a new proof
of the Stanley--Stembridge conjecture.

More specifically, Section~\ref{section:defs} provides definitions, Section~\ref{section:color} describes color matrices and their multiplication and trace properties, Section~\ref{section:tableau}
defines the tableau operators and states the main theorem, which is proved in Section~\ref{section:proof}. Section~\ref{section:trace} proves
the trace formula for tableau matrices and establishes $e$-positivity. Finally, Section~\ref{section:extra} describes the tableau matrices using Hikita's rule and proves the following explicit formula. The change-of-basis theorem, $e$-positivity results, and tableau formulas in this paper have been formalized in Lean~4 \cite{lean4}, using mathlib \cite{mathlib}. The source code, build instructions, and a correspondence between the statements of the paper and the Lean theorems are available in the \href{https://github.com/goombamaui/ProperCircularArcGraphs}{Lean formalization repository}.

\begin{theorem}\label{thm:pca-formula}
Let $\Gamma$ be a proper circular arc graph on $n$ vertices, obtained from
a natural unit interval graph $G$ on $[n+k]$ by identifying $n+r$ with $r$
for $1\le r\le k$ as in Lemma~\ref{lem:pca}, and let $N\ge n$. Then
\[
X_\Gamma(\bm x)=\sum_j\sum_{\substack{T\supseteq T_j\\
 T\setminus T_j=\{k+1,\dots,n+k\}\\ c_T(n+1,\dots,n+k)=j}}\
 \prod_{v=k+1}^{n+k}
 A\bigl(c_T(v)\mid c_T(N(v))\bigr)\frac{e_{c_T(v)}}{e_{c_T(v)-1}},
\]
where for sequences $\sigma$ of distinct positive integers,
\[
A(1\mid\sigma)=\prod_r\frac{\sigma_r-1}{\sigma_r},\qquad
A(\sigma_r+1\mid\sigma)=\frac{\sigma_r+1}{\sigma_r}\prod_{r'\ne r}\frac{\sigma_r-\sigma_{r'}+1}{\sigma_r-\sigma_{r'}},
\]
and otherwise $A(y\mid\sigma)=0$, $e_0=1$, $N(v)$ is the sequence of
neighbors of $v$ in $G$ smaller than $v$, and $c_T(v)$ is the column of
$v$ in $T$. The sum is over $k$-tuples $j$ of distinct elements of $[N]$
and tableaux $T$ with at most $N$ columns, drawn with rows increasing to
the right and columns increasing upward, where $T_j$ has rows of lengths
$j_1,\dots,j_k$, with $r$ at the end of the row of length $j_r$ and
distinct nonpositive entries elsewhere.
\end{theorem}

\section{Definitions}\label{section:defs}

Throughout, $N\in\bN\cup\{\infty\}$ is the number of colors, with $[N]=\{1,\ldots,N\}$ for $N\in\bN$ and $[\infty]=\bN$. Let $\bm x=(x_c)_{c\in[N]}$ be commuting variables, one for each color, and $e_j$ be the elementary symmetric functions in $\bm x$. In this section, all statements hold for every $N$.

For basis vectors, we suppress outer parentheses in tuple subscripts: $b_{a,w}$ means $b_{(a,w)}$, and likewise for $\hat b$. The empty tuple is denoted by $\emptyset$.

\begin{definition}
    A \emph{proper coloring} $\kappa: V(G) \to [N]$ assigns a number to each vertex so that adjacent vertices have different colors, meaning $\kappa(v) \neq \kappa(v')$ if $\{v, v'\}\in E(G)$. The \emph{chromatic symmetric function} \cite[Definition 2.1]{chromsym} is
    \[X_G(\bm x)=\sum_{\kappa \text{ proper}}\prod_{v\in V(G)}x_{\kappa(v)},\]
    where $\kappa$ ranges over proper colorings. For finite $N$, $X_G(x_1,\dots,x_N)$ is $X_G(\bm x)$ with $x_c=0$ for $c>N$, and when $N\ge|V(G)|$ this does not change the $e$-coefficients.
\end{definition}

\begin{definition}
A graph $G$ on $[n]$ is a \emph{natural unit interval graph} (NUIG) if $\{i, j\}\in E(G)$ with $i<j$ implies $\{i,k\}, \{k,j\}\in E(G)$ for all $i<k<j$. Equivalently, a graph is isomorphic to a NUIG if and only if it is the intersection graph of finitely many intervals on the real line, none of which contains another \cite{roberts}. A \emph{proper circular arc graph} is then the intersection graph of finitely many arcs of a circle, none of which contains another.

\end{definition}

\begin{definition}
A \emph{marked graph} $\bG=(G,\alpha,\beta)$ is a graph $G$ with two sequences of distinct vertices $\alpha$ (the \emph{outputs}) and $\beta$ (the \emph{inputs}); these are the bi-labeled graphs of \cite[Definition 3.1]{quantumiso}. If $\bH=(H,\alpha',\beta')$ and $\beta,\alpha'$ have the same length, then $\bG+\bH=\big((G\sqcup H)/(\beta_r\sim\alpha'_r),\ \alpha,\ \beta'\big)$ \cite[Definition 3.13]{quantumiso}. If $\alpha,\beta$ have the same length, the \emph{trace} $\operatorname{tr}\bG$ is the graph $G/(\alpha_r\sim\beta_r)$.
\end{definition}

\begin{example}
    Figure~\ref{fig:gluing} shows an example of $(G, (4,5), (1,2))+(H, (3,4), (1,2))$ forming a new NUIG; note that duplicate edges do not change proper colorings so we can ignore them.
\end{example}
\begin{figure}[ht]
\centering
\begin{tikzpicture}[every node/.style={font=\small}]
\colorlet{outG}{red}
\colorlet{glue}{blue}
\colorlet{inH}{green!55!black}
\def\h{.866}
\begin{scope}[shift={(0,0)}]
  \vertex[glue][-3pt][below]{1}[(0,0)]{(2.5,0)}[g]
  \vertex[glue][-3pt][below]{2}[(0,0)]{(1.5,0)}[g]
  \vertex[black][-3pt][above]{3}[(0,0)]{(1,\h)}[g]
  \vertex[outG][-3pt][below]{4}[(0,0)]{(0.5,0)}[g]
  \vertex[outG][-3pt][above]{5}[(0,0)]{(0,\h)}[g]
  \draw (vg1)--(vg2)--(vg3)--(vg4)--(vg5)--(vg3);
  \draw (vg2)--(vg4);
  \node at (1.25,-.95) {$\bG$};
\end{scope}
\node at (3.3,.43) {$+$};
\begin{scope}[shift={(4.1,0)}]
  \vertex[inH][-3pt][below]{1}[(0,0)]{(2,0)}[h]
  \vertex[inH][-3pt][above]{2}[(0,0)]{(1.5,\h)}[h]
  \vertex[glue][-3pt][below]{3}[(0,0)]{(1,0)}[h]
  \vertex[glue][-3pt][below]{4}[(0,0)]{(0,0)}[h]
  \draw (vh1)--(vh2)--(vh3)--(vh1);
  \draw (vh3)--(vh4);
  \node at (1,-.95) {$\bH$};
\end{scope}
\node at (6.85,.43) {$=$};
\begin{scope}[shift={(7.6,0)}]
  \vertex[inH][-3pt][below]{1}[(0,0)]{(3.5,0)}[s]
  \vertex[inH][-3pt][above]{2}[(0,0)]{(3,\h)}[s]
  \vertex[glue][-3pt][below]{3}[(0,0)]{(2.5,0)}[s]
  \vertex[glue][-3pt][below]{4}[(0,0)]{(1.5,0)}[s]
  \vertex[black][-3pt][above]{5}[(0,0)]{(1,\h)}[s]
  \vertex[outG][-3pt][below]{6}[(0,0)]{(0.5,0)}[s]
  \vertex[outG][-3pt][above]{7}[(0,0)]{(0,\h)}[s]
  \draw (vs1)--(vs2)--(vs3)--(vs1);
  \draw (vs3) to[bend left=14] (vs4);
  \draw (vs3) to[bend right=14] (vs4);
  \draw (vs5)--(vs4)--(vs6)--(vs5)--(vs7)--(vs6);
  \node at (1.75,-.95) {$\bG+\bH$};
\end{scope}
\end{tikzpicture}
\caption{The sum $(G,(4,5),(1,2))+(H,(3,4),(1,2))$.}
\label{fig:gluing}
\end{figure}
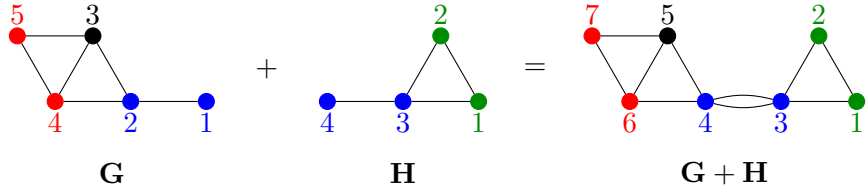
\begin{lemma}\label{lem:pca}
Every proper circular arc graph is $\operatorname{tr}\bG$ for some $\bG=(G,(n+1,\dots,n+k),(1,\dots,k))$, where $G$ is a NUIG on $[n+k]$ whose first $k$ and last $k$ vertices form cliques.
\end{lemma}

\begin{proof}
Using the arc representation of the proper circular arc graph, pick some point $p$ on the circle that is not an endpoint of an arc, and intersects $k$ intervals which we label $1, \ldots, k$ counterclockwise by their starting points. Then, cutting the circle at $p$ and unraveling the arcs onto the number line produces exactly the NUIG we need (after extending the terminal intervals to ensure no interval contains another), shown in Figure~\ref{fig:unroll}.
\end{proof}

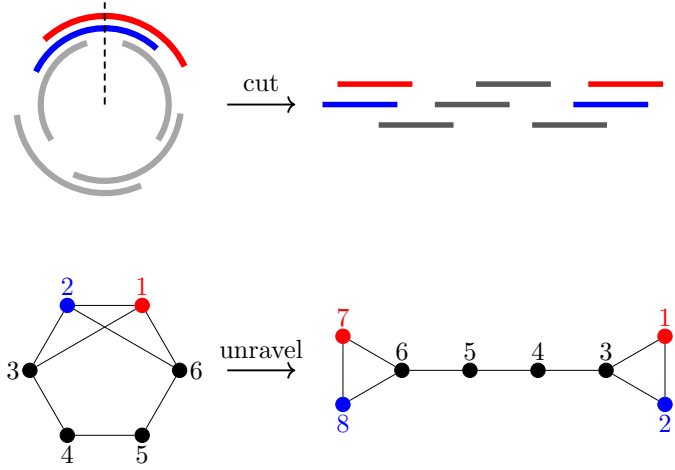
\begin{figure}[ht]
\centering

\scalebox{0.9}{
\begin{tikzpicture}[every node/.style={font=\small}]
\tikzmath{\dytop=0; \dybot=-3.9; \u=1.1; \C=3.35; \D=360/\C;}
\begin{scope}[shift={(1.5,\dytop)}]
  \foreach \a/\r/\c in {2.75/1.30/red, 2.95/1.12/blue, .15/.94/gray!70,
      .9/1.30/gray!70, 1.45/1.12/gray!70, 2.2/.94/gray!70}
    \draw[\c,line width=2.6pt] ({90+\D*\a}:\r) arc[start angle={90+\D*\a},delta angle=\D,radius=\r];
  \draw[densely dashed,thick] (0,0) -- (90:1.5);
\end{scope}
\draw[->,thick] (3.3,\dytop) -- (4.3,\dytop);
\node at (3.8,\dytop+.32) {cut};
\begin{scope}[shift={(4.7+.6*\u,\dytop)},x=\u cm]
  \foreach \b/\y/\c in {-.6/.3/red, .9/.3/gray!70!black, 2.75/.3/red,
      -.4/0/blue, 1.45/0/gray!70!black, 2.95/0/blue,
      .15/-.3/gray!70!black, 2.2/-.3/gray!70!black}
    \draw[\c,line width=2.3pt] ({3.35-\b},\y) -- ({2.35-\b},\y);
\end{scope}
\begin{scope}[shift={(1.5,\dybot)}]
  \vertex[red][-3pt][above]{1}[(0,0)]{(60:1.1)}
  \vertex[blue][-3pt][above]{2}[(0,0)]{(120:1.1)}
  \vertex[black][-3pt][left]{3}[(0,0)]{(180:1.1)}
  \vertex[black][-3pt][below]{4}[(0,0)]{(-120:1.1)}
  \vertex[black][-3pt][below]{5}[(0,0)]{(-60:1.1)}
  \vertex[black][-3pt][right]{6}[(0,0)]{(0:1.1)}
  \draw (v1)--(v2)--(v3)--(v4)--(v5)--(v6)--(v1)--(v3);
  \draw (v6)--(v2);
\end{scope}
\draw[->,thick] (3.3,\dybot) -- (4.3,\dybot);
\node at (3.8,\dybot+.32) {unravel};
\begin{scope}[shift={(5.0,\dybot)}]
  \tikzmath{\h=.866;}
  \vertex[red][-3pt][above]{1}[(0,0)]{(2*\h+3,.5)}
  \vertex[blue][-3pt][below]{2}[(0,0)]{(2*\h+3,-.5)}
  \vertex[black][-3pt][above]{3}[(0,0)]{(\h+3,0)}
  \vertex[black][-3pt][above]{4}[(0,0)]{(\h+2,0)}
  \vertex[black][-3pt][above]{5}[(0,0)]{(\h+1,0)}
  \vertex[black][-3pt][above]{6}[(0,0)]{(\h,0)}
  \vertex[red][-3pt][above]{7}[(0,0)]{(0,.5)}
  \vertex[blue][-3pt][below]{8}[(0,0)]{(0,-.5)}
  \draw (v1)--(v2)--(v3)--(v1);
  \draw (v3)--(v4)--(v5)--(v6);
  \draw (v6)--(v7)--(v8)--(v6);
\end{scope}
\end{tikzpicture}
}
\caption{A proper circular arc graph being unraveled into a NUIG.}
\label{fig:unroll}
\end{figure}

\section{Color matrices}\label{section:color}

\begin{definition}
Let $I_k=\{(i_1,\ldots,i_k)\in[N]^k:\ i_1,\ldots,i_k\text{ distinct}\}$. For finite $N$, the spaces of \emph{color vectors} $\hat W_k$ and $\hat V_k$ are spanned by the $\hat b_i$ with $i\in[N]^k$ and $i\in I_k$ respectively. For $N=\infty$, $\hat V_k$ consists of the infinite sums $\sum_{i\in I_k}F(x_{i_1},\dots,x_{i_k})\,\hat b_i$, where $F\in x_1\cdots x_k\,\bQ(e_1,e_2,\dots)[x_1,\dots,x_k]$. When $N=\infty$, the trace is the formal sum of the diagonal coefficients; our trace arguments use only finite $N$.
\end{definition}

\begin{definition}
The \emph{color matrix} of a marked graph $\bG=(G,\alpha,\beta)$ is $\hat M_\bG\colon\hat W_{\ell(\beta)}\to\hat W_{\ell(\alpha)}$,
\[
  \hat M_\bG \;=\; \sum_{\kappa\text{ proper}}
  \frac{\prod_{v\in V(G)} x_{\kappa(v)}}{\prod_{r} x_{\kappa(\beta_r)}}\,
  \hat b_{\kappa(\alpha)}\, \hat b_{\kappa(\beta)}^{T},
\]
where $\kappa(\alpha)=(\kappa(\alpha_1),\ldots,\kappa(\alpha_\ell))$. If $\alpha$ and $\beta$ are cliques, $\hat M_\bG$ restricts to a map $\hat V_{\ell(\beta)}\to\hat V_{\ell(\alpha)}$, which extends to $N=\infty$.
\end{definition}

\begin{example}
    Figure~\ref{fig:p3} shows the colors corresponding to $\hat M_{(P_3,(3),(1))}$ for $N=3$.
\end{example}

\begin{figure}[ht]
\centering
\begin{tikzpicture}[every node/.style={font=\small},
  col/.style={circle,draw,minimum size=11pt,inner sep=0pt,font=\scriptsize},
  c1/.style={col,fill=yellow!55},
  c2/.style={col,fill=cyan!35},
  c3/.style={col,fill=magenta!30}]
\begin{scope}[shift={(-4.75,-2.4)}]
  \vertex[blue][-3pt][below]{1}[(0,0)]{(1.6,0)}[p]
  \vertex[black][-3pt][below]{2}[(0,0)]{(.8,0)}[p]
  \vertex[red][-3pt][below]{3}[(0,0)]{(0,0)}[p]
  \draw (vp1)--(vp2)--(vp3);
  \node at (.8,-.75) {$(P_3,(3),(1))$};
\end{scope}
\def\cw{2.5} \def\ch{1.6}
\foreach \d in {1,2,3}
  \node at ({(\d-.5)*\cw},.35) {$\kappa(1)=\d$};
\foreach \c in {1,2,3}
  \node[anchor=east] at (-.15,{-(\c-.5)*\ch}) {$\kappa(3)=\c$};
\draw[gray!60] (0,0) grid[xstep=\cw,ystep=\ch] (3*\cw,-3*\ch);
\newcommand{\pcol}[5]{
  \node[c#1] (a) at ({#4+.55},{#5}) {#1};
  \node[c#2] (b) at ({#4},{#5}) {#2};
  \node[c#3] (c) at ({#4-.55},{#5}) {#3};
  \draw (a)--(b)--(c);}
\foreach \c/\d/\e in {1/2/3,1/3/2,2/1/3,2/3/1,3/1/2,3/2/1}{
  \pcol{\d}{\e}{\c}{(\d-.5)*\cw}{-(\c-.5)*\ch+.28}
  \node at ({(\d-.5)*\cw},{-(\c-.5)*\ch-.38}) {$x_{\c}x_{\e}$};}
\foreach \c/\e/\f in {1/2/3,2/1/3,3/1/2}{
  \pcol{\c}{\e}{\c}{(\c-.5)*\cw}{-(\c-.5)*\ch+.47}
  \pcol{\c}{\f}{\c}{(\c-.5)*\cw}{-(\c-.5)*\ch+.05}
  \node at ({(\c-.5)*\cw},{-(\c-.5)*\ch-.47}) {$x_{\c}(x_{\e}+x_{\f})$};}
\end{tikzpicture}
\caption{The color matrix of $(P_3,(3),(1))$ with $N=3$.}
\label{fig:p3}
\end{figure}

\begin{lemma}\label{lem:color}
$\hat M_{\bG+\bH}=\hat M_\bG \hat M_\bH$. If $\alpha$ and $\beta$ are disjoint cliques of the same length, then $\operatorname{tr}\hat M_\bG=X_{\operatorname{tr}\bG}(\bm x)$, the trace taken on $\hat V_{\ell(\beta)}$.
\end{lemma}

\begin{proof}
This is \cite[Lemma 3.21]{quantumiso} with weights: since a proper coloring of $\bG+\bH$ is a pair of proper colorings of $G$ and $H$ agreeing on the glued vertices,
\begin{gather*}
(\hat M_{\bG+\bH})_{i,j}=\sum_{\substack{\kappa(\alpha)=i\\\kappa(\beta')=j}}\ \prod_{v\notin\beta'}x_{\kappa(v)}
=\sum_l\Big(\sum_{\substack{\kappa(\alpha)=i\\\kappa(\beta)=l}}\ \prod_{v\notin\beta}x_{\kappa(v)}\Big)\Big(\sum_{\substack{\kappa'(\alpha')=l\\\kappa'(\beta')=j}}\ \prod_{v\notin\beta'}x_{\kappa'(v)}\Big)=(\hat M_\bG \hat M_\bH)_{i,j},\\
\operatorname{tr}\hat M_\bG=\sum_i\sum_{\kappa(\alpha)=\kappa(\beta)=i}\ \prod_{v\notin\beta}x_{\kappa(v)}=X_{\operatorname{tr}\bG}(\bm x).
\end{gather*}
\end{proof}

\begin{definition}
Let $K_m$ be the complete graph on $[m]$, and let
\[
\hat M_k=\hat M_{(K_{k+1},\,(1,\dots,k+1),\,(1,\dots,k))},\qquad \hat D_k=\hat M_{(K_k,\,(2,\dots,k),\,(1,\dots,k))},
\]
so that $\hat M_k\hat b_i=\sum_{c\notin i}x_c\,\hat b_{i,c}$ and $\hat D_k\hat b_i=\hat b_{i_2,\dots,i_k}$.
\end{definition}
These two matrices build every graph of Lemma~\ref{lem:pca}. More generally, if $G$ is a NUIG on $[n]$ whose first $a$ and last $b$ vertices form cliques, we call $\bG=(G,\alpha,\beta)$ with $\alpha=(n-b+1,\dots,n)$ and $\beta=(1,\dots,a)$ a \emph{marked NUIG}.

\begin{lemma}\label{lem:word}
If $\bG$ is a marked NUIG, then $\hat M_\bG=\hat X_1\cdots\hat X_m$, where each $\hat X_s$ is some $\hat M_j$ or $\hat D_j$. Hence, if $\bG$ is as in Lemma~\ref{lem:pca}, Lemma~\ref{lem:color} gives
\begin{equation}\label{eq:pcatrace}
X_{\operatorname{tr}\bG}(\bm x)=\operatorname{tr}(\hat{M}_{\bG}).
\end{equation}
\end{lemma}

\begin{proof}
We build $G$ one vertex at a time, starting with the inputs $1,\dots,a$ marked. By the NUIG condition, the earlier neighbors of each new vertex are the $r$ most recently marked vertices for some $r$, so we apply $\hat D$ to forget the older marked vertices and then $\hat M_r$ to attach the new vertex. Finally, we apply $\hat D$ until only the outputs $n-b+1,\dots,n$ remain. By Lemma~\ref{lem:color}, $\hat M_\bG$ is the product of these steps, with later steps on the left.
\end{proof}

\begin{example}\label{ex:word}
In Figure~\ref{fig:word}, the inputs $1,2$ are blue and the marked clique is red. The rightmost graph is built by $\hat M_3\hat M_2\hat D_3\hat M_2\hat D_3\hat M_2$. To add vertex $7$, whose earlier neighbors are $4,5,6$, we first forget vertex $3$ with $\hat D_4$ and then attach $7$ with $\hat M_3$, so that $4,\dots,7$ are marked.
\end{example}

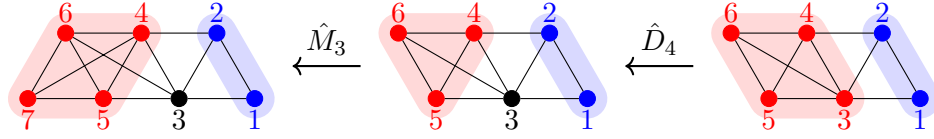
\begin{figure}[ht]
\centering
\begin{tikzpicture}[every node/.style={font=\small},
  hull/.style={line width=15pt,line join=round,line cap=round}]
\def\h{.866}
\colorlet{inshade}{blue!15}
\colorlet{retshade}{red!15}
\begin{scope}[shift={(9.3,0)}]
  \draw[hull,inshade,fill=inshade] (2.5,0)--(2,\h)--cycle;
  \draw[hull,retshade,fill=retshade] (1.5,0)--(0.5,0)--(0,\h)--(1,\h)--cycle;
  \vertex[blue][-3pt][below]{1}[(0,0)]{(2.5,0)}[a]
  \vertex[blue][-3pt][above]{2}[(0,0)]{(2,\h)}[a]
  \vertex[red][-3pt][below]{3}[(0,0)]{(1.5,0)}[a]
  \vertex[red][-3pt][above]{4}[(0,0)]{(1,\h)}[a]
  \vertex[red][-3pt][below]{5}[(0,0)]{(0.5,0)}[a]
  \vertex[red][-3pt][above]{6}[(0,0)]{(0,\h)}[a]
  \draw (va1)--(va2)--(va3)--(va4)--(va5)--(va6);
  \draw (va1)--(va3)--(va5) (va2)--(va4)--(va6) (va3)--(va6);
\end{scope}
\draw[->,thick] (8.8,.43) -- node[above] {$\hat D_4$} (7.9,.43);
\begin{scope}[shift={(4.9,0)}]
  \draw[hull,inshade,fill=inshade] (2.5,0)--(2,\h)--cycle;
  \draw[hull,retshade,fill=retshade] (0.5,0)--(0,\h)--(1,\h)--cycle;
  \vertex[blue][-3pt][below]{1}[(0,0)]{(2.5,0)}[b]
  \vertex[blue][-3pt][above]{2}[(0,0)]{(2,\h)}[b]
  \vertex[black][-3pt][below]{3}[(0,0)]{(1.5,0)}[b]
  \vertex[red][-3pt][above]{4}[(0,0)]{(1,\h)}[b]
  \vertex[red][-3pt][below]{5}[(0,0)]{(0.5,0)}[b]
  \vertex[red][-3pt][above]{6}[(0,0)]{(0,\h)}[b]
  \draw (vb1)--(vb2)--(vb3)--(vb4)--(vb5)--(vb6);
  \draw (vb1)--(vb3)--(vb5) (vb2)--(vb4)--(vb6) (vb3)--(vb6);
\end{scope}
\draw[->,thick] (4.4,.43) -- node[above] {$\hat M_3$} (3.5,.43);
\begin{scope}[shift={(0,0)}]
  \draw[hull,inshade,fill=inshade] (3,0)--(2.5,\h)--cycle;
  \draw[hull,retshade,fill=retshade] (1,0)--(0,0)--(0.5,\h)--(1.5,\h)--cycle;
  \vertex[blue][-3pt][below]{1}[(0,0)]{(3,0)}[c]
  \vertex[blue][-3pt][above]{2}[(0,0)]{(2.5,\h)}[c]
  \vertex[black][-3pt][below]{3}[(0,0)]{(2,0)}[c]
  \vertex[red][-3pt][above]{4}[(0,0)]{(1.5,\h)}[c]
  \vertex[red][-3pt][below]{5}[(0,0)]{(1,0)}[c]
  \vertex[red][-3pt][above]{6}[(0,0)]{(0.5,\h)}[c]
  \vertex[red][-3pt][below]{7}[(0,0)]{(0,0)}[c]
  \draw (vc1)--(vc2)--(vc3)--(vc4)--(vc5)--(vc6)--(vc7);
  \draw (vc1)--(vc3)--(vc5)--(vc7) (vc2)--(vc4)--(vc6) (vc3)--(vc6) (vc4)--(vc7);
\end{scope}
\end{tikzpicture}
\caption{Adding vertex $7$ with $\hat M_3\hat D_4$.}
\label{fig:word}
\end{figure}

\section{Tableau operators}\label{section:tableau}

Let $V_k$ be the $\bQ(e_1,e_2,\dots)$-vector space with basis $\{b_\sigma:\sigma\in I_k\}$. In Sections~\ref{section:tableau} and~\ref{section:proof}, $N=\infty$, so $I_k$ is the set of $k$-tuples of distinct positive integers.

\begin{definition}
For $\sigma\in I_k$, let $T(\sigma)=\{y\notin\sigma:\ y=1\text{ or }y-1\in\sigma\}$, and for $y\in T(\sigma)$ let
\[
A(1\mid\sigma)=\prod_r\frac{\sigma_r-1}{\sigma_r},\qquad A(\sigma_r+1\mid\sigma)=\frac{\sigma_r+1}{\sigma_r}\prod_{r'\ne r}\frac{\sigma_r-\sigma_{r'}+1}{\sigma_r-\sigma_{r'}},
\]
and $A(y\mid\sigma)=0$ otherwise. The \emph{tableau matrices} $M_k\colon V_k\to V_{k+1}$ and $D_k\colon V_k\to V_{k-1}$ are
\[
M_kb_\sigma=\sum_yA(y\mid\sigma)\,b_{\sigma,y},\qquad D_kb_\sigma=\frac{e_{\sigma_1}}{e_{\sigma_1-1}}\,b_{\sigma_2,\dots,\sigma_k}.
\]
\end{definition}

The entries of $M_k$ are nonnegative rationals, and those of $D_k$ are ratios $e_a/e_{a-1}$. If $a\notin\sigma$, then only the factor coming from $a$ is new, so for every $y\ne a+1$,
\begin{equation}\label{eq:ratio}
A(y\mid a,\sigma)=\frac{y-a}{y-a-1}\,A(y\mid\sigma).
\end{equation}

\begin{theorem}\label{thm:main}
There are linear maps $U_k\colon V_k\to\hat V_k$, $k\ge0$, with $U_0=1$, such that
\begin{align}
U_kM_{k-1}&=\hat M_{k-1}U_{k-1},\tag{M}\label{eq:intertwine-M}\\
\hat D_kU_k&=U_{k-1}D_k.\tag{D}\label{eq:intertwine-D}
\end{align}
\end{theorem}

\begin{definition}\label{def:tableau-matrix}
Let $\bG$ be a marked NUIG with $\hat M_\bG=\hat X_1\cdots\hat X_m$ as in Lemma~\ref{lem:word}. The \emph{tableau matrix} of $\bG$ is $M_\bG=X_1\cdots X_m$, where $X_s$ is $\hat X_s$ with the hat removed.
\end{definition}

Applying Theorem~\ref{thm:main} to each factor gives $U_bM_\bG=\hat M_\bG U_a$. Section~\ref{section:extra} describes $M_\bG$ directly using Hikita's tableaux.

\begin{example}\label{ex:p3}
Let $N=3$ and $\bG=(P_3,(3),(1))$. By Lemma~\ref{lem:word}, $\hat M_\bG=(\hat D_2\hat M_1)^2$ is the matrix in Figure~\ref{fig:p3}. On the tableau side, with rows and columns indexed by $b_1,b_2,b_3$ (dropping $b_4$, as in Section~\ref{section:trace}),
\[
D_2M_1=\begin{pmatrix}0&\frac{e_2}{2e_1}&\frac{2e_3}{3e_2}\\[2pt]2e_1&0&0\\[2pt]0&\frac{3e_2}{2e_1}&0\end{pmatrix},\qquad
(D_2M_1)^2=\begin{pmatrix}e_2&\frac{e_3}{e_1}&0\\[2pt]0&e_2&\frac{4e_1e_3}{3e_2}\\[2pt]3e_2&0&0\end{pmatrix}.
\]
The change of basis from Section~\ref{section:proof} is
\[
U_1=\begin{pmatrix}
x_1&\frac{x_1(x_2+x_3)}{2e_1}&\frac{e_3}{3e_2}\\[2pt]
x_2&\frac{x_2(x_1+x_3)}{2e_1}&\frac{e_3}{3e_2}\\[2pt]
x_3&\frac{x_3(x_1+x_2)}{2e_1}&\frac{e_3}{3e_2}
\end{pmatrix},
\]
and $U_1(D_2M_1)^2=\hat M_\bG U_1$. As a check,
\begin{align*}
D_1(D_2M_1)^2M_0&=\begin{pmatrix}e_1&\frac{e_2}{e_1}&\frac{e_3}{e_2}\end{pmatrix}\begin{pmatrix}e_2&\frac{e_3}{e_1}&0\\[2pt]0&e_2&\frac{4e_1e_3}{3e_2}\\[2pt]3e_2&0&0\end{pmatrix}\begin{pmatrix}1\\0\\0\end{pmatrix}\\
&=e_1e_2+3e_3=X_{P_3}(\bm x)=\hat D_1\hat M_\bG\hat M_0,
\end{align*}
and both traces equal $2e_2=X_{\operatorname{tr}\bG}(\bm x)$, where $\operatorname{tr}\bG$ is $P_2$.
\end{example}

\section{Proof of the main theorem}\label{section:proof}
We will use induction. To construct the inductive hypothesis, we define
\[
R_k=M_{k-1}D_k-D_{k+1}M_k,\qquad \hat R_k=\hat M_{k-1}\hat D_k-\hat D_{k+1}\hat M_k,
\]
and, for $1\le s<k$,
\[
S_sb_\sigma=\frac1{\sigma_{s+1}-\sigma_s}\,b_\sigma
+\left(1-\frac1{\sigma_{s+1}-\sigma_s}\right)b_{\tau_s\sigma},
\qquad \hat S_s\hat b_i=\hat b_{\tau_si},
\]
where $\tau_s(\sigma_1,\dots,\sigma_k)
=(\sigma_1,\dots,\sigma_{s-1},\sigma_{s+1},\sigma_s,\sigma_{s+2},\dots,\sigma_k)$ swaps the entries in positions $s$ and $s+1$.
In particular, $\hat{S}$ applies an adjacent swap and $\hat R_k\hat b_i=x_{i_1}\hat b_{i_2,\dots,i_k,i_1}$ applies a cyclic permutation. To simplify $R_k$, let $a\in\bN$ and $w\in I_{k-1}$, with
$a\notin w$, and let $(a,w)$ denote the tuple obtained by prepending $a$ to $w$. Direct substitution in the definition of $R_k$ yields
\begin{align}
R_kb_{a,w}
&=\frac{e_a}{e_{a-1}}\sum_y
  \bigl(A(y\mid w)-A(y\mid a,w)\bigr)b_{w,y}\notag\\
&=\frac{e_a}{e_{a-1}}\left(
  \left(\sum_{y\in T(w)}\frac{A(y\mid w)}{a+1-y}b_{w,y}\right)
  -A(a{+}1\mid a,w)b_{w,a+1}\right).
\label{eq:R}
\end{align}
We use the following identities for induction.

\begin{proposition}\label{prop:identities}
For $k\ge2$, the following hold for $\{M,D,R,S\}$ and $\{\hat{M},\hat{D},\hat{R},\hat{S}\}$:
\begin{align}
S_sM_{k-1}&=M_{k-1}S_s,\qquad S_sR_k=R_kS_{s+1} &&(1\le s\le k-2),\label{eq:swap-basic}\\
D_kR_k&=R_{k-1}D_kS_1,\label{eq:swap-delete}\\
S_{k-1}M_{k-1}R_{k-1}&=R_kM_{k-1},\qquad
S_{k-1}M_{k-1}M_{k-2}=M_{k-1}M_{k-2},\label{eq:last-swap}\\
S_{k-1}R_k^2&=R_k^2S_1.\label{eq:swap-square}
\end{align}
\end{proposition}

\begin{proof}
Equation~\eqref{eq:swap-basic} follows because $A(y\mid\sigma)$ depends only
on the set of entries of $\sigma$. For Equation~\eqref{eq:swap-delete},
let $\sigma\in I_{k-2}$ and let $a,b\notin\sigma$ be distinct, then
\begin{align*}
D_kR_kb_{a,b,\sigma}
 &=\frac{e_ae_b}{e_{a-1}e_{b-1}}\sum_{y\notin\sigma}
   \bigl(A(y\mid b,\sigma)-A(y\mid a,b,\sigma)\bigr)b_{\sigma,y}\\
 &=\frac{e_ae_b}{e_{a-1}e_{b-1}}\sum_{y\notin\sigma}
   \biggl(\frac{A(y\mid\sigma)-A(y\mid b,\sigma)}{b-a}\\
 &\qquad+\left(1-\frac1{b-a}\right)
   \bigl(A(y\mid\sigma)-A(y\mid a,\sigma)\bigr)\biggr)b_{\sigma,y}\\
 &=\frac1{b-a}\,\frac{e_a}{e_{a-1}}R_{k-1}b_{b,\sigma}
   +\left(1-\frac1{b-a}\right)\frac{e_b}{e_{b-1}}R_{k-1}b_{a,\sigma}\\
 &=R_{k-1}D_kS_1b_{a,b,\sigma}.
\end{align*}
For Equation~\eqref{eq:last-swap}, let $\sigma\in I_{k-2}$, let $y,z$ range
over distinct elements of $T(\sigma)$, and let
$P=\sum_yA(y\mid\sigma)A(y{+}1\mid\sigma,y)\,b_{\sigma,y,y+1}$, which
$S_{k-1}$ fixes, then
\begin{align*}
S_{k-1}M_{k-1}M_{k-2}b_\sigma
 &=S_{k-1}\Bigl(P+\sum_{y,z}\frac{z-y}{z-y-1}
   A(y\mid\sigma)A(z\mid\sigma)\,b_{\sigma,y,z}\Bigr)\\
 &=P+\sum_{y,z}A(y\mid\sigma)A(z\mid\sigma)
   \Bigl(\frac{b_{\sigma,y,z}}{z-y-1}+b_{\sigma,z,y}\Bigr)\\
 &=P+\sum_{y,z}\frac{z-y}{z-y-1}
   A(y\mid\sigma)A(z\mid\sigma)\,b_{\sigma,y,z}
 =M_{k-1}M_{k-2}b_\sigma.
\end{align*}
For its first identity, expanding the products gives
\begin{align*}
R_kM_{k-1}b_{a,\sigma}
&=\frac{e_a}{e_{a-1}}
 \sum_{\substack{y,z\notin\sigma\\y\ne z,\ y\ne a}}
 A(y\mid a,\sigma)
 \bigl(A(z\mid\sigma,y)-A(z\mid a,\sigma,y)\bigr)b_{\sigma,y,z}\\
&=S_{k-1}M_{k-1}R_{k-1}b_{a,\sigma}.
\end{align*}
For the second equality, if $y\ne a+1$ and $z\notin\{y+1,a+1\}$, Equation~\eqref{eq:ratio} gives both sides the coefficient $-\frac{e_a}{e_{a-1}}A(y\mid\sigma)A(z\mid\sigma)\frac{(y-a)(z-y)}{(y-a-1)(z-y-1)(z-a-1)}$ on $b_{\sigma,y,z}$. The remaining cases are similar.
The hatted identities follow directly by appending, deleting, and swapping
colors; for example,
\begin{align*}
\hat D_k\hat R_k\hat b_{a,b,\sigma}
 &=x_a\hat b_{\sigma,a}
  =\hat R_{k-1}\hat D_k\hat S_1\hat b_{a,b,\sigma},\\
\hat S_{k-1}\hat M_{k-1}\hat R_{k-1}\hat b_{a,\sigma}
 &=x_a\sum_{c\notin(a,\sigma)}x_c\hat b_{\sigma,c,a}
  =\hat R_k\hat M_{k-1}\hat b_{a,\sigma}.
\end{align*}
Finally, $S_s^2=1$ and $S_sD_{k+1}=D_{k+1}S_{s+1}$ directly from the
definitions. Equations~\eqref{eq:last-swap}, \eqref{eq:swap-delete}, and
\eqref{eq:swap-basic} now give (for the hatted identities as well)
\begin{align*}
S_{k-1}R_k^2
&=S_{k-1}R_kM_{k-1}D_k-S_{k-1}R_kD_{k+1}M_k\\
&=M_{k-1}R_{k-1}D_k-D_{k+1}M_kR_kS_1\\
&=M_{k-1}D_kR_kS_1-D_{k+1}M_kR_kS_1\\
&=R_k^2S_1,
\end{align*}
proving Equation~\eqref{eq:swap-square}.
\end{proof}

\begin{proof}[Proof of Theorem~\ref{thm:main}]
We construct $U_k$ satisfying
Equations~\eqref{eq:intertwine-M} and \eqref{eq:intertwine-D}, together with
\begin{align}
U_kR_k&=\hat R_kU_k,\tag{R}\label{eq:intertwine-R}\\
U_kS_s&=\hat S_sU_k\quad(1\le s<k).\tag{S}\label{eq:intertwine-S}
\end{align}
For $k=1$, let $e_r(\bm x\setminus x_c)$ denote the elementary symmetric function of degree $r$ in the variables $(x_d)_{d\ne c}$, and define
\[
U_1b_a=\frac1{a\,e_{a-1}}\sum_cx_c\,e_{a-1}(\bm x\setminus x_c)\,\hat b_c.
\]
With rows indexed by the $\hat b_c$ and columns by the $b_a$, its upper-left corner is
\[
U_1=
\begin{pmatrix}
x_1 & \dfrac{x_1(e_1-x_1)}{2e_1} & \cdots\\
x_2 & \dfrac{x_2(e_1-x_2)}{2e_1} & \cdots\\
\vdots & \vdots & \ddots
\end{pmatrix}.
\]
The definitions give
\[
R_1b_a=\frac{e_a}{e_{a-1}}\left(\frac1a b_1-\frac{a+1}{a}b_{a+1}\right),
\qquad \hat R_1\hat b_c=x_c\hat b_c.
\]
Together with
\[
e_a-e_a(\bm x\setminus x_c)=x_ce_{a-1}(\bm x\setminus x_c),
\qquad \sum_cx_ce_{a-1}(\bm x\setminus x_c)=ae_a,
\]
this yields
\begin{align*}
U_1M_0b_{\emptyset}
 &=U_1b_1=\sum_cx_c\hat b_c=\hat M_0U_0b_{\emptyset},\\
\hat D_1U_1b_a
 &=\frac1{a e_{a-1}}\sum_cx_ce_{a-1}(\bm x\setminus x_c)\,\hat b_{\emptyset}
 =\frac{e_a}{e_{a-1}}\hat b_{\emptyset}=U_0D_1b_a,\\
U_1R_1b_a
 &=\frac1{a e_{a-1}}\sum_cx_c^2e_{a-1}(\bm x\setminus x_c)\,\hat b_c
 =\hat R_1U_1b_a.
\end{align*}
Thus Equations~\eqref{eq:intertwine-M}, \eqref{eq:intertwine-D}, and
\eqref{eq:intertwine-R} hold, while Equation~\eqref{eq:intertwine-S} is
vacuous. The formula for $R_1b_a$, together with $b_1=M_0b_{\emptyset}$,
also gives $V_1=M_0V_0+R_1V_1$, the case $k=1$ of Equation~\eqref{eq:rank-span} below.

Now fix $k\ge2$ and assume the claims below level $k$. We construct $U_k$
by induction on its columns. Equation~\eqref{eq:intertwine-R} applied to $b_{a,w}$, expanded using Equation~\eqref{eq:R}, involves
$U_kb_{a,w}$, the columns $U_kb_{w,y}$ with $y\in T(w)$, and, when
$a+1\notin w$, the column $U_kb_{w,a+1}$; the $T(w)$-columns occur together
in Equation~\eqref{eq:intertwine-M}. Thus, by ordering the basis vectors
indexed by $I_k$, we can use the two equations inductively to construct
$U_k$. For $w\in I_{k-1}$
and $y\notin w$, give $b_{w,y}$ the rank $r(w,y)=(r_j(w,y))_{j\ge1}$, where
\[
r_j(w,y)=
\begin{cases}
k+1-s,&j=w_s,\\
1,&j=y,\ y\notin T(w),\\
0,&\text{otherwise}.
\end{cases}
\]
We compare $r(w,y)<r(w',y')$ if $r_j(w,y)<r_j(w',y')$ where $j$ is the largest differing index. Roughly, $r(w,y)<r(w',y')$ if $w'$ has larger entries than
$w$, and moving large entries to the right lowers the rank. Each rank has
finitely many predecessors, so we can fix $w\in I_{k-1}$ and assume inductively that
$U_k$ is defined on every column of smaller rank. The columns
$U_kb_{w,y}$ for $y\in T(w)$ have the same rank, and we solve for them
simultaneously. Let
$Z(w)=\{z\in w:z\ge2,\ z-1\notin w\}$. Equation~\eqref{eq:intertwine-M}
applied to $b_w$ and Equation~\eqref{eq:intertwine-R} applied to
$b_{z-1,w}$ for $z\in Z(w)$ give
\begin{gather}
\sum_{y\in T(w)}A(y\mid w)U_kb_{w,y}
 =\hat M_{k-1}U_{k-1}b_w,\label{eq:block-M}\\
\sum_{y\in T(w)}\frac{A(y\mid w)}{z-y}U_kb_{w,y}
 =\frac{e_{z-2}}{e_{z-1}}\hat R_kU_kb_{z-1,w}
\quad(\forall z\in Z(w)).\label{eq:block-R}
\end{gather}
The column $b_{z-1,w}$ has smaller rank than $b_{w,y}$ for $y\in T(w)$.
The right side of Equation~\eqref{eq:block-M} is known from $U_{k-1}$, and
those of Equation~\eqref{eq:block-R} from the induction on rank. Unpacking $T$ and $Z$ gives $|T(w)|=|Z(w)|+1$, so the system is square. To show the coefficient matrix is invertible, note $A(y\mid w) \neq 0$, so a null vector $(c_y)_{y\in T(w)}$ must satisfy
\[
\sum_{y\in T(w)}c_y=0,\qquad
\sum_{y\in T(w)}\frac{c_y}{z-y}=0\quad(\forall z\in Z(w)).
\]
For $z\in Z(w)$, note
\[\sum_{y\in T(w)}\frac{c_y}{z-y}=\frac{P(z)+\sum_{y\in T(w)}c_y z^{|T(w)|-1}}{Q(z)}=\frac{P(z)}{Q(z)}=0,\]
where $P$ has degree at most $|T(w)|-2$ and $Q(z)=\prod_{y\in T(w)}(z-y)\neq 0$ since $Z(w) \cap T(w)=\emptyset$.
Thus $P(z)=0$ for all $z\in Z(w)$, so $P=0$ by its degree bound. Plugging in $z=y\in T(w)$
\[
0=P(y)=\sum_{y'\in T(w)}c_{y'}\prod_{y''\in T(w)\setminus\{y'\}}(y-y'') = c_y\prod_{y'\in T(w)\setminus\{y\}}(y-y'),
\]
forcing $c_y=0$, meaning the only null vector is 0. Thus the system is invertible and determines $U_kb_{w,y}$ for $y\in T(w)$.

For the remaining $b_{w,y}$ with $y\notin T(w)$ and $y\notin w$, Equation~\eqref{eq:intertwine-R} gives
\begin{equation}\label{eq:free-R}
A(y\mid y-1,w)U_kb_{w,y}
 =\left(\sum_{z\in T(w)}\frac{A(z\mid w)}{y-z}U_kb_{w,z}\right)
  -\frac{e_{y-2}}{e_{y-1}}\hat R_kU_kb_{y-1,w}.
\end{equation}
The columns $U_kb_{w,z}$ were determined above, and $r(y-1,w)<r(w,y)$:
the largest entry of $(w,y)$ moves right if it lies in $w$, and disappears otherwise.

Thus $U_kb_{y-1,w}$ is known by induction. Since $A(y\mid y-1,w)\ne0$, this equation determines the remaining columns. Every instance of
Equations~\eqref{eq:intertwine-M} and \eqref{eq:intertwine-R} is used
exactly once, so both hold.
Inverting the coefficient matrix in Equations~\eqref{eq:block-M} and
\eqref{eq:block-R} also expresses $b_{w,y}$, $y\in T(w)$, as linear
combinations of $M_{k-1}b_w$ and $R_kb_{z-1,w}$, $z\in Z(w)$.
Equation~\eqref{eq:R} expresses each remaining $b_{w,y}$ using
$b_{w,z}$, $z\in T(w)$, and $R_kb_{y-1,w}$, whose source columns have
smaller rank. Thus, for every $r(w,y)=r$,
\begin{equation}\label{eq:rank-span}
\begin{aligned}
b_{w,y}&\in\operatorname{span}\bigl(\{M_{k-1}b_w\}
 \cup\{R_kb_{a,w}:a\in\bN\setminus w,\ a+1\in w\cup\{y\}\}\bigr)\\
&\subseteq M_{k-1}V_{k-1}
 +\operatorname{span}\{R_kb_\sigma:r(\sigma)<r\}.
\end{aligned}
\end{equation}

With $U_k$ constructed, we prove Equation~\eqref{eq:intertwine-S} by
induction on rank, keeping $k$ fixed. The inductive hypothesis at rank
$r$ states for every $b_\sigma$ with $r(\sigma)<r$,
\begin{align}
U_kS_sb_\sigma&=\hat S_sU_kb_\sigma\qquad(1\le s<k),\notag\\
U_kS_{k-1}R_kb_\sigma&=\hat S_{k-1}\hat R_kU_kb_\sigma.
\label{eq:intertwine-SR}
\end{align}
To establish the two equalities at rank $r$, Equation~\eqref{eq:rank-span}
reduces the check to 
\begin{align*}
    U_kS_sM_{k-1}&=\hat{S}_sU_kM_{k-1}, &U_kS_sR_kb_{\sigma} &= \hat{S}_sU_kR_kb_{\sigma} &(\forall r(\sigma)<r),\\
    U_kS_{k-1}R_kM_{k-1}&=\hat{S}_{k-1}\hat{R}_kU_kM_{k-1},
    &U_kS_{k-1}R_k^2b_{\sigma}&=\hat{S}_{k-1}\hat{R}_kU_kR_kb_{\sigma}&(\forall r(\sigma)<r).
\end{align*}
On $M_{k-1}V_{k-1}$, Equation~\eqref{eq:swap-basic} and induction on $k$
give, for $s\le k-2$,
\[
U_kS_sM_{k-1}=\hat S_sU_kM_{k-1}.
\]
For $s=k-1$, Equation~\eqref{eq:last-swap} and the established
Equations~\eqref{eq:intertwine-M} and \eqref{eq:intertwine-R} give
\begin{align*}
U_kS_{k-1}M_{k-1}M_{k-2}
 &=\hat S_{k-1}U_kM_{k-1}M_{k-2},\\
U_kS_{k-1}M_{k-1}R_{k-1}
 &=\hat S_{k-1}U_kM_{k-1}R_{k-1},
\end{align*}
which combined with Equation~\eqref{eq:rank-span} at level $k-1$ gives
$U_kS_{k-1}M_{k-1}=\hat S_{k-1}U_kM_{k-1}$.
The second required equality follows from
Equation~\eqref{eq:last-swap} and $S_{k-1}^2=\hat S_{k-1}^2=1$:
\[
U_kS_{k-1}R_kM_{k-1}
 =U_kM_{k-1}R_{k-1}
 =\hat M_{k-1}\hat R_{k-1}U_{k-1}
 =\hat S_{k-1}\hat R_kU_kM_{k-1}.
\]
On $R_kb_\sigma$ with $r(\sigma)<r$, the inductive hypothesis and
Equations~\eqref{eq:swap-basic} and \eqref{eq:swap-square} give the
remaining required equalities:
\begin{align*}
U_kS_sR_kb_\sigma
 &=\hat R_kU_kS_{s+1}b_\sigma
 =\hat S_sU_kR_kb_\sigma\qquad(1\le s\le k-2),\\
U_kS_{k-1}R_kb_\sigma
 &=\hat S_{k-1}\hat R_kU_kb_\sigma
 =\hat S_{k-1}U_kR_kb_\sigma,\\
U_kS_{k-1}R_k^2b_\sigma
 &=\hat R_k^2U_kS_1b_\sigma
 =\hat S_{k-1}\hat R_kU_kR_kb_\sigma,
\end{align*}
proving Equation~\eqref{eq:intertwine-S}. Finally, we prove Equation~\eqref{eq:intertwine-D} by a similar induction on rank,
keeping $k$ fixed. Fix a rank $r$. The inductive hypothesis is
\[
\hat D_kU_kb_\sigma=U_{k-1}D_kb_\sigma\qquad(\forall r(\sigma)<r).
\]
Yet again, to prove this equality for $b_{w,y}$ of rank $r$,
Equation~\eqref{eq:rank-span} reduces this to showing
\begin{align*}
\hat D_kU_kM_{k-1}&=U_{k-1}D_kM_{k-1},\\
\hat D_kU_kR_kb_{a,w}&=U_{k-1}D_kR_kb_{a,w}
\qquad(\forall a\in\bN\setminus w,\ a+1\in w\cup\{y\}).
\end{align*}
On $M_{k-1}V_{k-1}$, the definition of $R_{k-1}$ and induction on $k$
give
\begin{align*}
\hat D_kU_kM_{k-1}
 &=(\hat M_{k-2}\hat D_{k-1}-\hat R_{k-1})U_{k-1}\\
 &=U_{k-1}(M_{k-2}D_{k-1}-R_{k-1})
 =U_{k-1}D_kM_{k-1}.
\end{align*}
On $R_kb_{a,w}$, Equations~\eqref{eq:swap-delete},
\eqref{eq:intertwine-R}, and \eqref{eq:intertwine-S} give
\begin{align*}
\hat D_kU_kR_kb_{a,w}
 &=\hat R_{k-1}\hat D_kU_kS_1b_{a,w},\\
U_{k-1}D_kR_kb_{a,w}
 &=\hat R_{k-1}U_{k-1}D_kS_1b_{a,w}.
\end{align*}
The vector $S_1b_{a,w}$ only involves basis vectors of rank smaller than $r$: if $a+1=w_1$, then $S_1b_{a,w}=b_{a,w}$, and otherwise $\tau_1(a,w)=(w_1,a,w_2,\dots,w_{k-1})$ has smaller rank, since the entries after $w_1$ move right, $y$ drops out, and $a<a+1\in w\cup\{y\}$. Hence the inductive hypothesis makes the two displayed
right sides equal. Therefore
\begin{align*}
\hat D_kU_kR_kb_{a,w}&=U_{k-1}D_kR_kb_{a,w}
\qquad(\forall a\in\bN\setminus w,\ a+1\in w\cup\{y\}).
\end{align*}
This completes the induction, proving Equation~\eqref{eq:intertwine-D} and Theorem~\ref{thm:main}.
\end{proof}

\section{Traces and \texorpdfstring{$e$}{e}-positivity}\label{section:trace}

We have found a change of basis family $U_k$ that relates the color matrices to tableau matrices. To transport the trace property from coloring to tableau, we need to show that when restricted to a finite yet sufficient amount of colors, the maps in Theorem~\ref{thm:main} are invertible.

We fix $N\in\bN$ and work over $\bQ(x_1,\dots,x_N)$, so $I_k$ is now
the set of $k$-tuples of distinct entries in $[N]$, and $V_k$ has basis
$\{b_\sigma:\sigma\in I_k\}$.
Let $M_k^{(N)}$ be the restriction of $M_k$ to these tuples, omitting any
term $b_{\sigma,N+1}$, and let $D_k$ act by the same formula as before,
with $e_j=e_j(x_1,\dots,x_N)$. We want to find a map $U_k^{(N)}$ that is invertible with
\[
U_k^{(N)}M_{k-1}^{(N)}=\hat M_{k-1}U_{k-1}^{(N)},
\qquad
\hat D_kU_k^{(N)}=U_{k-1}^{(N)}D_k.
\]
It turns out $U_k^{(N)}$ comes from setting $x_c=0$ for $c>N$ in $U_k$. We
need to show this is well-defined, that its image lies in $\hat V_k$, and
that it satisfies the two equations above.

First, the restriction is well-defined. Let $\sigma$ be a $k$-tuple of
distinct positive integers. By the definition of $\hat V_k$ with
$N=\infty$, $U_kb_\sigma$ decomposes as
\[
U_kb_\sigma=\sum_iF_\sigma(x_{i_1},\dots,x_{i_k})\hat b_i,\]
where we will show $F_\sigma(x_{i_1},\dots,x_{i_k})\in x_{i_1}\cdots x_{i_k}\,\bQ[\bm x][e_j^{-1}:j<\max\sigma]$.
This holds for $U_1$ and is preserved by $\hat M_{k-1}$ and
$\hat R_k$, which multiply each new color by its variable. By induction on
rank, Equations~\eqref{eq:block-M} and~\eqref{eq:block-R} divide only by
$e_j$ with $j<\max w$, and Equation~\eqref{eq:free-R} only by $e_j$ with
$j<\max(w,y)$.

Thus if $\sigma$ has entries at most $N+1$, no $e_j$ with
$j>N$ appears in a denominator, and we may set $x_c=0$ for $c>N$. This
defines $U_k^{(N)}b_\sigma$ for every $\sigma$ with entries at most $N+1$. In particular, $U_k^{(N)}$ is defined on $V_k$.

Second, the image lies in $\hat V_k$, since $x_{i_1}\cdots x_{i_k}$ divides
each $F_\sigma(x_{i_1},\dots,x_{i_k})$, so every term with a color outside
$[N]$ vanishes.

Third, specializing Equation~\eqref{eq:intertwine-D} gives
$\hat D_kU_k^{(N)}=U_{k-1}^{(N)}D_k$. Specializing
Equation~\eqref{eq:intertwine-M} at $b_w$ with $w\in I_{k-1}$ gives
\[
U_k^{(N)}M_{k-1}b_w=\hat M_{k-1}U_{k-1}^{(N)}b_w,
\qquad
M_{k-1}b_w=M_{k-1}^{(N)}b_w+A(N{+}1\mid w)\,b_{w,N+1}.
\]
We use the following lemma to show $U_k^{(N)}b_{w,N+1}=0$, so
$U_k^{(N)}M_{k-1}^{(N)}=\hat M_{k-1}U_{k-1}^{(N)}$.

\begin{lemma}\label{lem:vanish}
If $\sigma$ has entries at most $N+1$ and contains $N+1$, then $U^{(N)}_kb_\sigma=0$.
\end{lemma}

\begin{proof}
We use downward induction on $k$. For $k=N+1$, there are no color tuples
with distinct entries, so $\hat V_{N+1}=0$. Now suppose $k\le N$ and
the claim holds at level $k+1$. Since $\sigma$ contains $N+1$, there is
an $a\in[N]\setminus\sigma$. Equation~\eqref{eq:intertwine-D} gives
\[
0=\hat D_{k+1}U_{k+1}^{(N)}b_{a,\sigma}
 =\frac{e_a}{e_{a-1}}U_k^{(N)}b_\sigma,
\]
proving the claim.
\end{proof}

From now on, $U_k^{(N)}$ denotes its restriction to $V_k$, so
$U_k^{(N)}\colon V_k\to\hat V_k$ satisfies the two equations above.

\begin{proposition}\label{prop:finite}
For $0\le k\le N$, the map
$U_k^{(N)}\colon V_k\to\hat V_k$ is invertible. For $1\le k\le N$,
\[
U_k^{(N)}M_{k-1}^{(N)}=\hat M_{k-1}U_{k-1}^{(N)},
\qquad
\hat D_kU_k^{(N)}=U_{k-1}^{(N)}D_k.
\]
\end{proposition}

\begin{proof}
Equations~\eqref{eq:intertwine-M} and \eqref{eq:intertwine-D}, together with
Lemma~\ref{lem:vanish}, give the two displayed equalities. Next, since
$\hat R_k=\hat M_{k-1}\hat D_k-\hat D_{k+1}\hat M_k$, note
\[
\hat R_kU_k^{(N)} = U_k^{(N)}(M_{k-1}^{(N)}D_k-D_{k+1}M_k^{(N)}),
\]
where $M_N^{(N)}=\hat M_N=0$. This means $\hat{R}_k$ preserves the image of $U_k^{(N)}$.

We prove invertibility by showing $U_k^{(N)}$ is surjective by induction on $k$, starting with $U_0^{(N)}=1$. The inductive hypothesis then states that $U_{k-1}^{(N)}$ is surjective, so for every $i\in I_{k-1}$ we can find a $v$ such that $\hat b_i=U_{k-1}^{(N)}v$. Then $\hat M_{k-1}\hat b_i=U_k^{(N)}M_{k-1}^{(N)}v$ lies in the
image of $U_k^{(N)}$, and since applying $\hat{R}_k$ preserves the image, so does
\[
(\hat R_k^{k})^u\hat M_{k-1}\hat b_i
 =(x_{i_1}\cdots x_{i_{k-1}})^u
   \sum_{c\in[N]\setminus i}x_c^{u+1}\hat b_{i,c}
 \qquad(0\le u\le N-k).
\]
For fixed $i$, the coefficient matrix is a Vandermonde matrix in the
$N-k+1$ distinct variables $x_c$, up to nonzero row and column factors.
These vectors therefore span every $\hat b_{i,c}$. This proves
surjectivity, and both spaces have dimension $N!/(N-k)!$, so
$U_k^{(N)}$ is invertible.

\end{proof}

For a marked NUIG $\bG$, let $M_\bG^{(N)}$ be $M_\bG$ with each $M_j$
replaced by $M_j^{(N)}$.

\begin{theorem}\label{thm:pca}
Let $\Gamma=\operatorname{tr}\bG$ be a proper circular arc graph on $n$
vertices, where $\bG=(G,(n+1,\dots,n+k),(1,\dots,k))$ as in
Lemma~\ref{lem:pca}. For any $N\in\bN$,
\[
X_\Gamma(x_1,\dots,x_N)=\operatorname{tr}M_\bG^{(N)},
\]
which is $e$-positive. In particular, $\Gamma$ is $e$-positive.
\end{theorem}

\begin{proof}
If every level in the product defining $M_\bG^{(N)}$ is at most $N$, then
Proposition~\ref{prop:finite} applies to each factor, and with
Lemma~\ref{lem:word}, which holds for every $N$,
\[
\operatorname{tr}M_\bG^{(N)}
=\operatorname{tr}\bigl((U_k^{(N)})^{-1}\hat M_\bG U_k^{(N)}\bigr)
=\operatorname{tr}\hat M_\bG
=X_\Gamma(x_1,\dots,x_N).
\]
Otherwise, the product passes through $V_j$ and $\hat V_j$ for some
$j>N$, which are both zero since $[N]$ has no $j$ distinct colors, so
$\operatorname{tr}M_\bG^{(N)}=0=\operatorname{tr}\hat M_\bG
=X_\Gamma(x_1,\dots,x_N)$.

The entries of $M_j^{(N)}$ and $D_j$ are Laurent polynomials in $e_1,\dots,e_N$ with nonnegative coefficients, so the trace is too. Since $e_1,\dots,e_N$ are algebraically independent and the trace is the polynomial $X_\Gamma(x_1,\dots,x_N)$, it is $e$-positive. Taking $N\ge n$ shows $X_\Gamma(\bm x)$ is $e$-positive.
\end{proof}

\begin{example}\label{ex:pca}
The proper circular arc graph $\Gamma$ in Figure~\ref{fig:unroll} is
$\operatorname{tr}\bG$ for $\bG=(G,(7,8),(1,2))$, where $G$ is the NUIG
on $[8]$ shown there. Building $G$ as in Lemma~\ref{lem:word} gives
$\hat M_\bG=\hat D_3\hat M_2(\hat M_1\hat D_2)^4\hat D_3\hat M_2$, so by
Theorem~\ref{thm:pca} and cyclicity of the trace,
\[
X_\Gamma(x_1,\dots,x_N)
=\operatorname{tr}\bigl((M_1^{(N)}D_2)^4(D_3M_2^{(N)})^2\bigr)
=16e_1e_5+20e_2e_4+12e_3^2+84e_6.
\]
\end{example}

When $k=0$, we do not need Proposition~\ref{prop:finite}, since $U_0=1$.

\begin{theorem}\label{thm:stanstem}
If $G$ is a NUIG, then
\[
X_G(\bm x)=M_{(G,\emptyset,\emptyset)},
\]
which is $e$-positive. This proves the $e$-positivity of unit interval graphs without the modular law characterization of \cite{chrommodularlaw} used in \cite{stanstemproof}, and together with \cite[Theorem 5.1]{stanstemreduction} gives an alternate proof of the Stanley--Stembridge conjecture.
\end{theorem}

\begin{proof}
Since $U_0=1$,
\[
X_G(\bm x)=\hat M_{(G,\emptyset,\emptyset)}U_0=U_0M_{(G,\emptyset,\emptyset)}=M_{(G,\emptyset,\emptyset)}.
\]
As in the proof of Theorem~\ref{thm:pca}, this is a Laurent polynomial in the algebraically independent $e_1,e_2,\dots$ with nonnegative coefficients that is also a polynomial, so it is $e$-positive.
\end{proof}

\begin{example}\label{ex:stanstem}
For the NUIG $G+H$ in Figure~\ref{fig:gluing}, building it as in
Lemma~\ref{lem:word} and applying Theorem~\ref{thm:stanstem} gives
\begin{align*}
X_{G+H}(\bm x)&=D_1D_2D_3M_2D_3M_2M_1D_2M_1D_2D_3M_2M_1M_0\\
&=8e_1e_3^2+88e_3e_4+24e_2e_5+40e_1e_6+56e_7,
\end{align*}
which is $e$-positive since every factor has entries that are Laurent
polynomials in the $e_i$ with nonnegative coefficients.
\end{example}

\section{Tableau matrices of graphs}\label{section:extra}

We now explain $M_k$ and $D_k$ using Hikita's tableaux, which gives a
tableau matrix $M_\bG$ for every marked NUIG. This generalizes
\cite[Section 4]{singlegluing} to several input and output boxes, with a
different normalization. Throughout, $N=\infty$.

We draw tableaux with rows increasing to the right and columns
increasing upward.

\begin{definition}[{\cite[Definitions 1.2 and 1.5]{stanstemproof}}]\label{def:hikita-rule}
For a tableau $T$, let $c_T(v)$ be the column of the box $v$, and for a
sequence of vertices $\alpha$, let
$c_T(\alpha)=(c_T(\alpha_1),\dots,c_T(\alpha_\ell))$. Let
$N(v)$ be the sequence of smaller neighbors of $v$ in increasing order.
\emph{Hikita's rule} at $q=1$ adds $v$ to the top of column $y$ of $T$
with probability $A(y\mid c_T(N(v)))/y$, and we write
$T\leftarrow_y v$ for the result. For a NUIG $G$, adding the remaining
vertices of $G$ in increasing order builds $T'$ from $T$ with probability
\[
\Pr(T'\mid T,G)=\prod_{v\in T'\setminus T}
 \frac{A\bigl(c_{T'}(v)\mid c_{T'}(N(v))\bigr)}{c_{T'}(v)}.
\]
Let
$\Pr(T'\mid G)=\Pr(T'\mid\emptyset,G)$.
\end{definition}

This is Hikita's transition probability $\varphi_y$ at $q=1$ \cite[Definitions 1.2 and 1.5]{stanstemproof}, where his red columns are the columns $\le0$ and those containing a neighbor of $v$: telescoping the factors of $A$ over each run of consecutive entries of $c_T(N(v))$ gives his formula. Since the smaller neighbors of $v$ are at the tops of distinct columns, $v$ can be added on top of column $y$ whenever $A(y\mid c_T(N(v)))\ne0$.

\begin{theorem}[{\cite[Theorem 1.6]{stanstemproof} at $q=1$}]\label{thm:hikita}
For a NUIG $G$,
\[
X_G(\bm x)=\sum_T\Pr(T\mid G)\prod_r\lambda_r!\,e_{\lambda_r},
\]
where $T$ has shape $\lambda$.
\end{theorem}

We combine the probabilities and the elementary symmetric functions into
a single weight.

\begin{definition}\label{def:weight}
For a NUIG $G$ and $T,T'$ as in Definition~\ref{def:hikita-rule}, the
\emph{weight} of $T'$ is
\[
w(T'\mid T,G)=\prod_{v\in T'\setminus T}
 A\bigl(c_{T'}(v)\mid c_{T'}(N(v))\bigr)
 \frac{e_{c_{T'}(v)}}{e_{c_{T'}(v)-1}}.
\]
\end{definition}

Adding a box in column $y$ turns a row of length $y-1$ into one of length
$y$, so the factors $ye_y/e_{y-1}$ along each row telescope, and
\begin{equation}\label{eq:weight-pr}
w(T'\mid T,G)=\Pr(T'\mid T,G)\,
 \frac{\prod_r\lambda'_r!\,e_{\lambda'_r}}{\prod_r\lambda_r!\,e_{\lambda_r}},
\end{equation}
where $\lambda$ and $\lambda'$ are the shapes of $T$ and $T'$.

\begin{definition}\label{def:graph-tableau}
Let $\bG=(G,\alpha,\beta)$ be a marked NUIG, with $\beta=(1,\dots,a)$ and
$\alpha=(n-b+1,\dots,n)$. For $j\in I_a$, let $T_j$ be a tableau with
rows of lengths $j_1,\dots,j_a$ and $c_{T_j}(\beta)=j$ (whose other
entries are distinct and nonpositive), where $T_\emptyset=\emptyset$. With
$\varepsilon(\sigma)=\prod_re_{\sigma_r}/e_{\sigma_r-1}$, let
$\mathcal M_\bG\colon V_a\to V_b$ be
\[
\mathcal M_\bG b_j=\sum_{T}w(T\mid T_j,G)\,
 \frac{\varepsilon(j)}{\varepsilon(c_T(\alpha))}\,b_{c_T(\alpha)},
\]
where $T$ ranges over tableaux containing $T_j$ whose other entries are
$a+1,\dots,n$, with $c_T(\alpha)\in I_b$. Since the weight
only depends on the columns of the vertices of $G$, the choice of
nonpositive entries does not matter.
\end{definition}

\begin{lemma}\label{lem:tableau-mult}
If $\bG=(G,\alpha,\beta)$ and $\bH=(H,\alpha',\beta')$ are marked NUIGs
and $\beta,\alpha'$ have the same length, then
$\mathcal M_{\bG+\bH}=\mathcal M_\bG\mathcal M_\bH$.
\end{lemma}

\begin{proof}
Building $G+H$ from $T_j$ first adds the vertices of $H$, giving some
$T$, and then those of $G$, giving some $T'$. Each $N(v)$ is the same in $G+H$ as in $G$, so
$w(T'\mid T_j,G+H)=w(T\mid T_j,H)\,w(T'\mid T,G)$, and the second factor
is the same as building $G$ from $T_i$, where $i=c_T(\alpha')$. Grouping
by $i$,
\[
\mathcal M_{\bG+\bH}b_j
=\sum_i\Bigl(\sum_{c_T(\alpha')=i}w(T\mid T_j,H)\,
 \frac{\varepsilon(j)}{\varepsilon(i)}\Bigr)\mathcal M_\bG b_i
=\mathcal M_\bG\mathcal M_\bH b_j.
\]
\end{proof}

\begin{proposition}\label{prop:tableau-word}
The marked cliques satisfy
\[
\mathcal M_{(K_{k+1},(1,\dots,k+1),(1,\dots,k))}=M_k,
\qquad
\mathcal M_{(K_k,(2,\dots,k),(1,\dots,k))}=D_k.
\]
\end{proposition}

\begin{proof}
This follows by unfolding the definitions of $w$ and $\varepsilon$.
\end{proof}

\begin{theorem}\label{thm:tableau}
For every marked NUIG $\bG$, $M_\bG=\mathcal M_\bG$. For $N\in\bN$,
$M_\bG^{(N)}=\mathcal M_\bG^{(N)}$, where $\mathcal M_\bG^{(N)}$ is
$\mathcal M_\bG$ with $j$ and $T$ restricted to at most $N$ columns and
$e_i=e_i(x_1,\dots,x_N)$.
\end{theorem}

\begin{proof}
As in the proof of Lemma~\ref{lem:word}, $\bG$ is a sum of the marked
cliques in Proposition~\ref{prop:tableau-word}, so this follows from
Lemma~\ref{lem:tableau-mult}. The finite case is proved the same way.
\end{proof}

Taking $a=b=0$, Theorems~\ref{thm:stanstem} and~\ref{thm:tableau} give
$X_G(\bm x)=\mathcal M_{(G,\emptyset,\emptyset)}=\sum_Tw(T\mid\emptyset,G)$,
which by Equation~\eqref{eq:weight-pr} gives a new proof of
Theorem~\ref{thm:hikita}.

On the diagonal of $M_\bG^{(N)}$, the factors $\varepsilon$ cancel. This
gives an explicit formula for proper circular arc graphs: unroll the
graph as in Lemma~\ref{lem:pca}, place the inputs in columns $j$, build
the rest, and keep the tableaux whose outputs return to the columns $j$.

\begin{restated}{thm:pca-formula}
Let $\Gamma=\operatorname{tr}\bG$ be a proper circular arc graph on $n$
vertices with $\bG=(G,(n+1,\dots,n+k),(1,\dots,k))$ as in
Lemma~\ref{lem:pca}, and let $N\ge n$. Then
\[
X_\Gamma(\bm x)=\sum_j\sum_{\substack{T\supseteq T_j\\
 T\setminus T_j=\{k+1,\dots,n+k\}\\ c_T(n+1,\dots,n+k)=j}}\
 \prod_{v=k+1}^{n+k}
 A\bigl(c_T(v)\mid c_T(N(v))\bigr)\frac{e_{c_T(v)}}{e_{c_T(v)-1}},
\]
where $j$ ranges over $I_k$ with entries at most $N$, and $T$ ranges
over tableaux with at most $N$ columns.
\end{restated}

\begin{proof}
By Theorems~\ref{thm:pca} and~\ref{thm:tableau}, $X_\Gamma(x_1,\dots,x_N)$
is the sum of the diagonal entries $w(T\mid T_j,G)$ with
$c_T(\alpha)=j$, which is the right side. Since $N\ge n$, both sides have
the same expansion in the algebraically independent $e_1,\dots,e_N$, so
the identity holds in infinitely many variables.
\end{proof}

By Equation~\eqref{eq:weight-pr}, each summand in
Theorem~\ref{thm:pca-formula} equals
\[
\Pr(T\mid T_j,G)\,
 \frac{\prod_r\lambda_r!\,e_{\lambda_r}}{\prod_rj_r!\,e_{j_r}},
\]
where $\lambda$ is the shape of $T$, so Theorem~\ref{thm:pca-formula} is
Theorem~\ref{thm:hikita} with the first and last $k$ vertices tied
together.

\begin{example}\label{ex:pca-formula}
Let $\Gamma=P_2$, so $k=1$ and $\bG=(P_3,(3),(1))$ as in Example~\ref{ex:p3}. With $N=3$, Figure~\ref{fig:tableau-example} shows every tableau built from $T_1,T_2,T_3$ with at most $3$ columns, with the factor of each added vertex on its arrow. Each final tableau $T$ is labeled by $w(T\mid T_j,G)\,\varepsilon(j)/\varepsilon(c_T(3))\,b_{c_T(3)}$, so summing the labels for each $j$ gives column $j$ of $\mathcal M_\bG^{(3)}=(D_2M_1)^2$. Its diagonal comes from the two tableaux where vertex $3$ returns to column $j$, so $X_{P_2}(\bm x)=2e_2$.
\end{example}

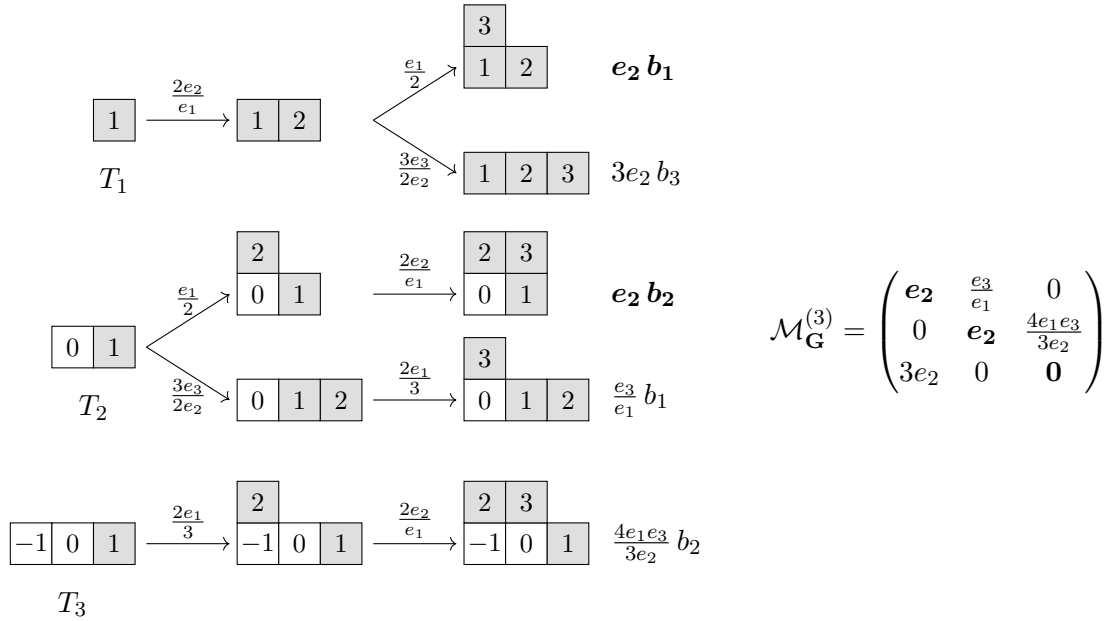
\begin{figure}[ht]
\centering
\begin{tikzpicture}[
  every node/.style={font=\small},
  build/.style={->},
  wt/.style={font=\footnotesize,inner sep=2pt}
]
  \def\s{.55}
  \newcommand{\T}[3]{%
    \foreach \c/\r/\e/\f in {#3}{%
      \ifnum\f=1 \def\bxfill{gray!25}\else\def\bxfill{white}\fi
      \draw[fill=\bxfill] ({#1+(\c-1)*\s},{#2-\s/2+(\r-1)*\s})
        rectangle ++(\s,\s);
      \node[font=\footnotesize] at ({#1+(\c-.5)*\s},{#2+(\r-1)*\s}) {$\e$};
    }}
  \def\ha(#1,#2)(#3){\draw[build] ({#1+1.8},#2) -- ({#1+2.9},#3)}
  \T{1.1}{0}{1/1/1/1}
  \T{3}{0}{1/1/1/1,2/1/2/1}
  \T{6}{.7}{1/1/1/1,2/1/2/1,1/2/3/1}
  \T{6}{-.7}{1/1/1/1,2/1/2/1,3/1/3/1}
  \ha(0,0)(0) node[midway,above,wt] {$\frac{2e_2}{e_1}$};
  \ha(3,0)(.7) node[midway,above,wt] {$\frac{e_1}{2}$};
  \ha(3,0)(-.7) node[midway,below,wt] {$\frac{3e_3}{2e_2}$};
  \node[anchor=west] at (7.8,.7) {$\bm{e_2\,b_1}$};
  \node[anchor=west] at (7.8,-.7) {$3e_2\,b_3$};
  \T{.55}{-3}{1/1/0/0,2/1/1/1}
  \T{3}{-2.3}{1/1/0/0,2/1/1/1,1/2/2/1}
  \T{3}{-3.7}{1/1/0/0,2/1/1/1,3/1/2/1}
  \T{6}{-2.3}{1/1/0/0,2/1/1/1,1/2/2/1,2/2/3/1}
  \T{6}{-3.7}{1/1/0/0,2/1/1/1,3/1/2/1,1/2/3/1}
  \ha(0,-3)(-2.3) node[midway,above,wt] {$\frac{e_1}{2}$};
  \ha(0,-3)(-3.7) node[midway,below,wt] {$\frac{3e_3}{2e_2}$};
  \ha(3,-2.3)(-2.3) node[midway,above,wt] {$\frac{2e_2}{e_1}$};
  \ha(3,-3.7)(-3.7) node[midway,above,wt] {$\frac{2e_1}{3}$};
  \node[anchor=west] at (7.8,-2.3) {$\bm{e_2\,b_2}$};
  \node[anchor=west] at (7.8,-3.7) {$\frac{e_3}{e_1}\,b_1$};
  \T{0}{-5.6}{1/1/-1/0,2/1/0/0,3/1/1/1}
  \T{3}{-5.6}{1/1/-1/0,2/1/0/0,3/1/1/1,1/2/2/1}
  \T{6}{-5.6}{1/1/-1/0,2/1/0/0,3/1/1/1,1/2/2/1,2/2/3/1}
  \ha(0,-5.6)(-5.6) node[midway,above,wt] {$\frac{2e_1}{3}$};
  \ha(3,-5.6)(-5.6) node[midway,above,wt] {$\frac{2e_2}{e_1}$};
  \node[anchor=west] at (7.8,-5.6) {$\frac{4e_1e_3}{3e_2}\,b_2$};
  \node at (1.375,-.8) {$T_1$};
  \node at (1.1,-3.8) {$T_2$};
  \node at (.825,-6.4) {$T_3$};
  \node[anchor=west] at (9.9,-2.8) {$\displaystyle
    \mathcal M_\bG^{(3)}=\begin{pmatrix}
      \bm{e_2}&\frac{e_3}{e_1}&0\\[2pt]
      0&\bm{e_2}&\frac{4e_1e_3}{3e_2}\\[2pt]
      3e_2&0&\bm{0}
    \end{pmatrix}$};
\end{tikzpicture}
\caption{Every build from $T_1,T_2,T_3$ with at most $3$ columns, giving
the columns of $\mathcal M_\bG^{(3)}=(D_2M_1)^2$ for
$\bG=(P_3,(3),(1))$.}
\label{fig:tableau-example}
\end{figure}

\section{Further directions}\label{section:conclusions}

In this paper, we proved the $e$-positivity of proper circular arc graphs and provided an alternate proof of the $e$-positivity of unit interval graphs by proving that color matrices and tableau matrices are related by a change of basis.

Our proof constructs $U_k$ recursively but does not give an explicit formula. In particular, the matrix can be computationally calculated but the formula for each entry becomes very complicated at $k\geq 3$. Finding an explicit formula for $U_k$ and its inverse in finitely many variables would provide a more direct way to calculate tableau matrices and color matrices from one another.

In particular, for any non-NUIG marked graph $\bG$ with clique inputs and outputs, Proposition~\ref{prop:finite} still defines a matrix $M_{\bG}=(U_b^{(N)})^{-1}\hat M_\bG U_a^{(N)}$, providing a linear extension of tableau to other graphs. Note $M_{\bG}$ need not be built from $M_k$ and $D_k$, so its entries may not be $e$-positive. Can these entries be described by a rule on tableaux when the graph is not a NUIG? For example, a rule for attaching a cycle could lead to further $e$-positivity results for graphs glued along cliques, extending the
single-vertex constructions in \cite{singlegluing}.

Another direction is to seek a $q$-refinement of these tableau matrices, using the same basis to describe the effect of ascent weights. A corresponding trace formula would extend the $q$-trace question in \cite[Conjecture 74]{singlegluing} to gluing along larger cliques. While the $e$-coefficients of
$X_G(\bm x;q)$ are now known to lie in $\bN[q]$ for every natural unit
interval graph \cite{openaichrom}, the corresponding positivity
question for proper circular arc digraphs remains open for general $q$. Additionally, Hwang noted that his construction extends to the circular
setting \cite[Section~5.1]{hwangflip}, and Alexandersson and Dai proved
the LLT analogue of his refinement for circular graphs \cite[Theorem~A]{lltflip}. It would be
interesting to know whether the formula of Theorem~\ref{thm:pca-formula}
admits a corresponding refinement.

Similarly, one can extend the tableau formulas to the
Tutte symmetric function in this basis, extending the matrix
construction in \cite{tuttematrix}. Since Tutte weights allow adjacent
vertices to have the same color, such an extension would also need to
account for repeated colors on the marked vertices. This could perhaps extend Hikita's tableau method to provide a simpler formula for the Tutte symmetric function of NUIGs.

\section{Acknowledgments}
The author thanks Per Alexandersson for pointing out the related circular $e$-positivity conjectures and the recent work on circular LLT polynomials and Hwang's refinement, and for suggestions on the Lean formalization.

During the preparation of this manuscript, the author used Claude Opus 5.0, Opus 5.5, and GPT-6-Sol for exploratory discussion, computationally checking conjectures, and improving prose. All arguments presented above were independently verified by the author who assumes full responsibility for the mathematics in the paper.

\end{document}